\documentclass[12pt,reqno]{amsart}
\usepackage{amsmath,amsthm,amsfonts,amssymb}
\usepackage[english]{babel}
\usepackage{tikz}
\usepackage{fourier}
\usepackage{graphicx}
\usepackage{epsfig}

\numberwithin{equation}{section}

\newcommand{\Fix}{\operatorname{Fix}}

\newcommand{\dist}{\operatorname{dist}}

\newcommand{\diff}{\operatorname{Diff}}

\newcommand{\al} {\alpha}       
        
\newcommand{\ga} {\gamma}       \newcommand{\Ga}{\Gamma}

      \newcommand{\La}{\Lambda}

\newcommand{\s} {\sigma}

\newcommand{\per}{\text{Per}}

\newcommand{\tf}{\tilde f}

\newcommand{\tM}{\tilde M}
\newcommand{\Fs}{\mathcal{F}^{s}}
\newcommand{\Fu}{\mathcal{F}^{u}}
\newcommand{\Fc}{\mathcal{F}^{c}}

\newcommand{\Fcs}{\mathcal{F}^{cs}}
\newcommand{\Fcu}{\mathcal{F}^{cu}}

\newcommand{\tFc}{\tilde {\mathcal{F}^{c}}}

\def \NN {{\mathbb N}}

\def \DD {{\mathbb D}}
\def \RR {{\mathbb R}}
\def \ZZ {{\mathbb Z}}

\def \T {{\mathbb T}}

\def \cO {{\mathcal O}}

\newtheorem{theorem}{Theorem}[section]

\newtheorem{proposition}[theorem]{Proposition}

\newtheorem{lemma}[theorem]{Lemma}

\newtheorem{definition}[theorem]{Definition}

\theoremstyle{remark}
\newtheorem{remark}[theorem]{Remark}

\begin{document}

\renewcommand{\subjclassname}{\textup{2000} Mathematics Subject Classification}


\setcounter{tocdepth}{2}



\renewcommand{\subjclassname}{\textup{2000} Mathematics Subject Classification}


\title{Ergodicity and Partial Hyperbolicity on Seifert manifolds}

\begin{abstract} We show that conservative partially hyperbolic diffeomorphism isotopic to the identity on Seifert 3-manifolds are ergodic.

\end{abstract}

\thanks{JRH and RU are partially supported by NSFC 11871262.
This research was partially supported by the
Australian Research Council.}
\author{Andy Hammerlindl}
\address{School of Mathematical Sciencies, Monash University, Victoria 3800 Australia}
\email{andy.hammerlindl@monash.edu}

\author{Jana Rodriguez Hertz}
\address{Department of Mathematics, Southern University of Science and Technology, 1088 Xueyuan Rd., Xili, Nanshan District, Shenzhen, Guangdong, China 518055}
\email{rhertz@sustc.edu.cn}

\author{Ra\'ul Ures}
\address{Department of Mathematics, Southern University of Science and Technology, 1088 Xueyuan Rd., Xili, Nanshan District, Shenzhen, Guangdong, China 518055} \email{ures@sustc.edu.cn}

\maketitle
\section{Introduction}

Since the foundational work of Grayson, Pugh, and Shub,
a large focus of the study of partially hyperbolic dynamics
has been to determine which of these systems are ergodic \cite{gps}.
It is conjectured that ``most'' volume preserving partially hyperbolic
diffeomorphisms are ergodic,
and indeed, when the center is one-dimensional,
it is known that the subset of ergodic diffeomorphism
is $C^1$-open and $C^\infty$-dense \cite{hhu}.
Moreover, for any dimension of center,
the set of ergodic diffeomorphisms contains
a $C^1$-open and $C^1$-dense subset of the space of
all volume preserving partially hyperbolic diffeomorphisms \cite{acw}.

In the three-dimensional setting
where each of the stable $E^s$, unstable $E^u$, and center $E^c$
bundles is one-dimensional,
it is further conjectured that there is a unique
obstruction to ergodicity:
the presence of embedded tori tangent to the $E^u \oplus E^s$ direction
\cite{chhu}.
We call such a torus a $us$-\emph{torus}.
This conjecture has been verified for
nilmanifolds \cite{rru_nilman} and solvmanifolds \cite{ha}.
Further, only certain manifolds support a $us$-torus \cite{hhu3}
and for any system with a $us$-torus
an exact description of the possible ergodic decompositions is known \cite{ha}.
For derived-from-Anosov systems on the 3-torus,
the question was partially answered in \cite{hu}
and recently S.~Gan and Y.~Shi have announced that all
of these systems are ergodic \cite{gash}.

To verify the conjecture,
it is necessary to show in dimension three
that any system without a $us$-torus is ergodic.
Here, we consider this question in the setting
of Seifert manifolds and prove the following.


\begin{theorem}\label{maintheorem} Let $f:M\to M$ be a $C^{2}$ conservative partially hyperbolic diffeomorphism on a 3-dimensional Seifert manifold $M$ such that $f$ is isotopic to the identity. Then, $f$ is accessible, hence ergodic. 
\end{theorem}

In this setting of Seifert manifolds,
Barthelm\'e, Fenley, Frankel, and Potrie
have announced
that any partially hyperbolic diffeomorphism
isotopic to the identity is leaf conjugate
to a topological Anosov flow \cite{bffp}.
Using this result would simplify some parts of our proof.
However, we present in this paper
a proof of Theorem \ref{maintheorem}
which relies only on existing results already in the literature.

Further,
Fenley and Potrie have adapted
techniques in \cite{bffp} to the study of accessibility
classes and announced that for any closed 3-mani\-fold
with hyperbolic geometry, every volume preserving
partially hyperbolic diffeomorphism
is accessible and ergodic \cite{fenpot}.

\smallskip{}

For Theorem \ref{maintheorem}, we split the proof into two
cases depending on wheth\-er or not the diffeomorphism has periodic points.
Note that in our setting, there are diffeomorphisms without periodic points.
For instance, one can consider an Anosov flow on a Seifert manifold
and take the time-$t$ map where $t$ is not the period of any periodic orbit.
Most of the arguments in the setting of no periodic points
are general and do not make use of the fact that we are on a Seifert
manifold. Because of this, we are also able to establish a 
theorem which holds for general 3-manifolds.

\begin{theorem}\label{noperiodicpoints} Let $f$ be a $C^2$ conservative partially hyperbolic diffeomorphism on a $3$-dimensional manifold $M$. If $f$ has a compact periodic center leaf and no periodic points, then $f$ is ergodic. 
\end{theorem}

The assumption of a compact periodic center leaf is likely unnecessary.  R.~Saghin and J.~Yang have announced the following result:
\begin{quote}\label{periodicleaf}
    Let $f\in\diff^{2}_{m}(M)$ be partially hyperbolic, where $M$ is a
    $3$-manifold. If $\per(f)=\varnothing$,  then $f$ has a compact periodic
    center manifold \cite{sy}.
\end{quote}
However, as a proof of this result is not yet available, we state Theorem \ref{noperiodicpoints} as above.

Finally, we consider the family of hyperbolic orbifolds
which consist of a sphere with exactly three cone points added.
These orbifolds are small enough that their mapping class group
is trivial, and using Theorem \ref{maintheorem}
we show that any partially hyperbolic diffeomorphism
on the unit tangent bundle of such an orbifold is ergodic.
Hence, this gives another family of manifolds
where the above conjecture of ergodicity is established.

\begin{theorem}\label{turnover}
    Let $\Sigma$ be an orbifold which is a sphere with exactly
    three cone points added.
    Then any partially hyperbolic diffeomorphism $f$
    defined on the unit tangent bundle of $\Sigma$
    has an iterate which lifts to a map
    $f_1$ defined on a circle bundle and which is isotopic to the identity.
    In particular,
    if $f$ is volume preserving,
    then it is accessible and ergodic.
\end{theorem}

\section{Preliminaries}\label{section.preliminaries}

\subsection{Partial hyperbolicity}

We will denote by $\diff^{r}(M)$ the set of $C^{r}$ diffeomorphisms of $M$, and by $\diff^{r}_{m}(M)$ the set of conservative $C^{r}$ diffeomorphisms of $M$.
\begin{definition} A diffeomorphism $f\in\diff^{1}(M)$ of a closed manifold $M$ is {\em partially
hyperbolic} if the tangent bundle $TM$ of $M$, splits into three invariant sub-bundles: 
$TM=E^{s}\oplus E^{c}\oplus E^{u}$
such that  all unit vectors
$v^\s\in E^\s_x$ ($\s= s, c, u$) with $x\in M$ satisfy  :
\begin{equation}\label{pointwise.ph}
 \|T_xfv^s\| < \|T_xfv^c\| < \|T_xfv^u\| 
\end{equation}
for some suitable Riemannian metric. The {\em stable bundle} $E^{s}$ must also satisfy
$\|Tf|_{E^s}\| < 1$ and the {\em unstable bundle}, $\|Tf^{-1}|_{E^u}\| < 1$. The bundle $E^{c}$ is called the {\em center bundle}.

We write $E^{cu}=E^c\oplus E^u$ and $E^{cs}=E^c\oplus E^s$

\end{definition}

\begin{remark}\label{remark.orientability}
    In proving accessibility or ergodicity,
    we may freely lift $f$ to a finite covering or replace it by an
    iterate.
    Therefore, we assume throughout the paper that all of the
    invariant bundles are oriented and that $f$ preserves these orientations.
\end{remark}

We assume from now on that $f$ is partially hyperbolic on a three dimensional manifold.
It is a well-known fact that the {\em strong} bundles, $E^s$ and $E^u$, are
uniquely integrable \cite{bp,hps}.  That is, there are invariant strong
foliations ${\mathcal W}^{s}(f)$ and ${\mathcal W}^{u}(f)$ tangent,
respectively,  to the invariant bundles $E^{s}$ and $E^{u}$.  The
integrability of $E^{c}$ is a more delicate matter.  We call
${\mathcal W}^{\sigma}(f)$  any foliation tangent to $E^\sigma$, $\sigma=s,u,
c, cs, cu$, whenever it exists and $W^\sigma_f(x)$ the leaf of ${\mathcal
W}^{\sigma}(f)$ passing through $x$. A subset $\Lambda$ is {\em
$\sigma$-saturated} if $W^\sigma_f(x) \subset \Lambda$
for every $x\in \Lambda$. A closed
$\sigma$-saturated subset $K$ is {\em minimal} if
$\overline{W^\sigma_f(x)}=\Lambda$ for every $x\in \Lambda$.
We say that a
foliation is minimal if $M$ is a minimal set for it; that is, every leaf is
dense.
An invariant foliation is $f$-\emph{minimal}
if the orbit of every leaf is dense in $M$.

For each $x\in M$, the {\em accessibility class} $AC(x)$ of $x$ is the
set of points $y\in M$ that can be joined to $x$ by a path that is piecewise
tangent to $E^{s}\cup E^{u}$.
Define
\begin{equation}\label{lamination}
 \Gamma(f)=\{x: AC(x)\text{ is not open }\}
\end{equation}
Then $\Gamma(f)$ is a closed invariant set laminated by accessibility classes. The lamination  $\Gamma(f)$ is $C^{1}$ \cite{ha, hhu}.
If $f$ has a single accessibility class $AC(x) = M$, then we
say that $f$ is \emph{accessible}.

\begin{proposition}\cite[Proposition 3.4]{rru_nilman}\label{sublamination} If $\Lambda\subset\Gamma(f)$ is $f$-invariant and $us$-saturated, and $\varnothing\ne \Lambda\ne M$, then $\Lambda\cap\per(f)\ne\varnothing$. Moreover, the boundary leaves of $\Lambda$ have Anosov dynamics and dense periodic points with the intrinsic topology.
\end{proposition}

The proof of the following proposition can be found in Lemma 3.3 and Remark 3.4 of \cite{hhu2}.

\begin{proposition}\cite{hhu2} \label{acc.transitivity} If $\Omega(f)=M$  and there exists $x$ such that the orbit of $AC(x)$ is dense, then $f$ is transitive. 
\end{proposition}

\begin{theorem} \cite{hhu2} \label{integrable} If $\Omega(f)=M$ and $\per(f)=\varnothing$, then $E^{c}$ is uniquely integrable. 
\end{theorem}

Let us state some results which will be needed in Section \ref{section.horizontal}. A consequence of \cite[Theorem 1.1]{hhu3} is:
\begin{theorem}\cite{hhu3}\label{anosov.tori} If  $\Gamma(f)$ has a compact leaf, then $M$ is one of the following 3-manifolds:
\begin{enumerate}
 \item the 3-torus
 \item the mapping torus of $-id:\T^{2}\to\T^{2}$
 \item the mapping torus of a hyperbolic automorphism on $\T^{2}$
\end{enumerate}
\end{theorem}
Indeed, the hypothesis in \cite[Theorem 1.1]{hhu3} is that $M$ is an irreducible manifold admitting an {\em Anosov torus}. A manifold $M$ admits an Anosov torus if there exists an invariant $2$-torus with a hyperbolic dynamics which can be extended to a diffeomorphism in the ambient manifold $M$. Observe that any compact leaf of $\Gamma(f)$ must be a 2-torus, since the leaves of $\Gamma(f)$ are foliated by the one-dimensional unstable foliation. Take the set $\Lambda$ of points in $\Gamma(f)$ belonging to compact leaves. By \cite{haefliger}, $\Lambda$ is closed. $\Lambda$ is also invariant. If $\Lambda=M$, then $M$ is foliated by 2-tori, and the result above trivially follows. If $\Lambda\ne M$, then Proposition \ref{sublamination} implies there exists a periodic torus with Anosov dynamics, hence $M$ admits an Anosov torus. On the other hand, if $M$ admits a partially hyperbolic dynamics, then $M$ is irreducible \cite[Theorem 1.3]{hhu3}.

\subsection{Circle bundles over hyperbolic surfaces}

Although our results are formulated for Seifert manifolds with hyperbolic orbifold, we do not lose generality if we assume that $M$ is a circle bundle over a surface with negative Euler characteristic, see Remark \ref{remark.orientability}. Indeed, it is well known that every Seifert manifold with hyperbolic orbifold is finitely covered by a circle bundle. See \cite[\S 7]{haposh} and references therein.

A circle bundle is a manifold $M$ having a smooth map $p:M\to \Sigma$ such that for every $x\in \Sigma$, the fiber $p^{-1}(\{x\})$ is a circle and there is an open set $U$, $x\in U$, such that $p^{-1}(U)$ is diffeomorphic to $U\times  S^1$ via a fiber preserving diffeomorphism.

A leaf of a foliation $\mathcal F$ is \emph{vertical} if it contains every fiber it intersects. A leaf $L$  of $\mathcal F$ is \emph{horizontal} if it is transverse to the fibers. We say that $\mathcal F$ is horizontal if all its leaves are horizontal. 

\begin{theorem}[Brittenham-Thurston \cite{br}]
Let $\mathcal F$ be a foliation without torus leaves in a Seifert space $M$. Then, there is an isotopy $\psi_t:M\to M$ from the identity such that the foliation $\psi_1(\mathcal F)$ satisfies that every leaf is either everywhere transverse to the fibers (horizontal) or saturated by fibers (vertical).
\end{theorem}

Of course, if is possible to work inversely and apply the isotopy to $p$ in order to have that $\mathcal F$ has the property that every leaf is vertical or horizontal with respect to $p\circ \psi_1$.

For more details see \cite[Section 2]{haposh}.
%
%

\section{No periodic points}\label{section.noperiodicpoints}

In this section, we prove Theorem \ref{noperiodicpoints}.
Assume throughout the section that $f\in \diff^{2}_{m}(M)$ is partially hyperbolic and has no periodic points.
By assumption, $f$ has a compact periodic center manifold $\gamma$.
By replacing $f$ by a iterate, we may assume $f(\gamma) = \gamma$.
Note by Theorem \ref{integrable} that $f$ is dynamically coherent. 
In this setting, accessibility implies ergodicity
\cite{bw, hhu}, so we further assume that $f$ is not accessible.

\begin{lemma}\label{usdense}
    The lamination $\Ga(f)$ is in fact an $f$-minimal foliation and $f$ is transitive.
\end{lemma}
\begin{proof}
By Proposition \ref{sublamination}, $\Gamma(f)=M$ and the orbit of any $us$-leaf is dense. Indeed, if it were not dense its closure would have periodic points. Transitivity then follows from Proposition \ref{acc.transitivity}.
\end{proof}

%
%

For a point $x \in M$, the \emph{center (Lyapunov) exponent}
is given by the limit
$\chi^c(x) = \lim_{N \to \infty} \sum_{n = 0}^{N-1} \log ||Df(f^n(x))|_{E^c}||$.
By Oseledets' Theorem, this limit exists $\mu$-almost everywhere
for any invariant measure $\mu$.

\begin{lemma}\label{centerzero}
For any invariant measure $\mu$,
the center exponent is zero $\mu$-almost everywhere. 
\end{lemma}

\begin{proof}
    Suppose an invariant measure has a non-zero center exponent on a set of
    positive measure.
    Then, one of the measures in its ergodic decomposition would have the same
    property.
    Hence, we may assume there is an ergodic measure with a non-zero center
    exponent.
    This measure is then hyperbolic and implies that there are
    periodic points \cite{ka}.
    This gives a contradiction.
\end{proof}

\begin{lemma} \label{lemma.cero}
The functions
    $\phi_n : M \to \mathbb{R},
    \ x \mapsto \tfrac1n \log||Df^n_x |_{E^c}||$
converge uniformly to zero. 
\end{lemma}

\begin{proof}
    Write $\phi = \phi_1$ and note that since
    the center is one-dimensional,
    $\phi_n(x) = \tfrac{1}{n} \sum_{j=1}^{n-1} \phi(f^j(x)).$
    Assume $\phi_n$ does not converge uniformly to zero.
    Then there are integers $n_k$ tending to $+\infty$
    and points $x_k \in M$ such that
    $\phi_{n_k}(x_k) \ge \alpha > 0$ for all $k$.
    Define measures
    $\mu_k = \tfrac{1}{n_k} \sum_{j=0}^{n_k-1} \delta_{f^j(x_k)}.$
    By restricting to a subsequence,
    we may assume these measures converge in the weak sense to
    a limit measure $\mu$ which one may show is $f$-invariant.
    Then,
    $\int \phi \, d\mu
    = \lim_{k \to \infty} \int \phi \, d\mu_k
    = \lim_{k \to \infty} \phi_{n_k}(x_n)
    \ge \alpha$,
    a contradiction.
\end{proof}


\begin{proof}[Proof of Theorem \ref{noperiodicpoints}]
    Let $\Lambda$ be a measurable $us$-saturated subset of $M$.
    As shown in \cite{bw, hhu},
    to establish ergodicity of $f$,
    it is enough to show that $\Lambda$ has either
    full or null Lebesgue measure.
    The intersection $\Lambda\cap \gamma$ is $f$-invariant.
    By the previous lemma,
    $f^n$ is 2-normally hyperbolic for large $n$
    and this implies that $\gamma$ is $C^2$ \cite{hps}.
    Since the rotation number of $f|_\gamma$ is irrational and $C^2$
    we may apply \cite[Th\'eor\`eme 1.4]{her} and get that
    $\Lambda\cap \gamma$ has full or null Lebesgue measure in $\gamma$. Since
    $\Gamma(f)$ is $f$-minimal and absolutely continuous, $\Lambda$ has full or
    null measure.
    This implies the ergodicity of $f$. 
 \end{proof}
 
 
 
 \section{$\Gamma(f)$ is horizontal}\label{section.horizontal}
 
 From now on $M$ will be a 3-dimensional Seifert manifold with negative Euler characteristic orbifold.
 
 \begin{proposition}\label{ushorizontal} Let $f:M\to M$ be a non accessible partially hyperbolic diffeomorphism. Then, $\Gamma(f)$ is horizontal. 
 \end{proposition}
 
 \begin{proof}
By Theorem \ref{anosov.tori}, $\Gamma(f)$ cannot have compact leaves.
Let $V(f)$ be the union of vertical leaves in $\Gamma(f)$.
This set $V(f)$ is closed and $f$-invariant. 
If $V(f)=M$, then all leaves would be vertical, and then the orbifold would have a one-dimensional foliation, which is impossible. Now, since there are no compact leaves, all leaves in $V(f)$ are cylinders. If $V(f)\ne \varnothing$, then by Proposition \ref{sublamination}, there would be a periodic leaf with Anosov dynamics and dense periodic points. The cylinder does not support such dynamics, as we prove for completeness in Lemma \ref{lemma.cylinder} below.
 \end{proof}

\begin{remark}\label{foliation.with.vertical.leaves} The argument in the proof above shows that, in fact, no codimension one foliation can be such that all its leaves are vertical. 
\end{remark}

\begin{lemma}\label{lemma.cylinder} A cylinder does not support Anosov dynamics with dense periodic points.
\end{lemma}
\begin{proof}
We adapt an argument given in \cite{franks}.
The unstable manifold of each periodic point is dense. Indeed, it is enough to see that $\overline{W^{u}(p)}$ is open, hence it is the whole cylinder. Take $x\in\overline{W^{u}(p)}$, and take a small neighborhood $U$ of $x$. Take a periodic point $q$ in $U$. By local product structure, the stable manifold of $q$ intersects the unstable manifold of $x$. Since $x$ is in the closure of $W^{u}(p)$, the stable manifold of $q$ intersects $W^{u}(p)$. Hence, $q\in\overline{ W^{u}(p)}$. Therefore since periodic  points are dense in $U$, $U\subset\overline{ W^{u}(p)}$. The same proof shows that each half-line $W^{u}_{\pm}(p)$ is dense in the cylinder.

Since $W^{u}_{+}(p)$ is dense, it intersect the local stable manifold of $p$ in another point. There are two possibilities. Either the unstable-stable loop encloses a disc, which would produce a contradiction by Poincar\'e-Ben\-dixon, or else it cuts the cylinder into two connected components. By orientability, the rest of the demi-line has to remain in only one connected component, which contradicts its density. 
\end{proof}
 
 \section{Invariance of the leaves of horizontal foliations}\label{section.horizontal}
  
Suppose $h : M \to M$ is a homeomorphism and $\hat M$ is covering space for $M$.
In general, there may be multiple possible lifts of $h$ to a homeomorphism on $\hat M$.
We call $\hat h : \hat M \to\hat M$ a \emph{good lift} if it commutes with all of the deck transformations
of the covering.
Note that if $h$ is isotopic to the identity, then it has at least one good lift,
since we may lift the isotopy to the covering space.
We use this property to prove the following.

 \begin{theorem}\label{leafinvariance} Let $M$ be a non-trivial circle bundle over a surface $\Sigma$ of negative Euler characteristic, $h:M \to M$ a homeomorphism isotopic to the identity, and $\Lambda$ an $h$-invariant horizontal minimal lamination. Then, there exists $k\in \NN$ such that every leaf of the lamination is $h^k$-invariant. 
 \end{theorem}
 
Since the lamination is horizontal,
the connected components of $M \setminus \Lambda$
are $I$-bundles where the fibers come from the circle
fibering on $M$.
Because of this, we may extend $\Lambda$
to a horizontal foliation $\mathcal F$
defined on all of $M$.
We may further assume that $\mathcal F$
is invariant by $h$.

There is a correspondence between horizontal foliations
on circle bundles and the action of surface groups on $S^1$.
See, for instance, \cite[\S2]{mann}.
Because of this, we will reformulate and prove
the result purely in the language of group actions.

A \emph{surface group} is the fundamental group of a closed surface of genus $g \ge 2$.
This group has generators $\alpha_i, \beta_i$
for $i = 1, \ldots, n$ with the relation
\[
    [\alpha_1, \beta_1]
    \cdots
    [\alpha_g, \beta_g]
    =
    1.
\]
When a surface group acts on $S^1$,
we may lift the generators
$\alpha_i, \beta_i : S^1 \to S^1$
to maps 
$\tilde \alpha_i, \tilde \beta_i : \RR \to \RR$
where $\RR$ is the universal cover of
$S^1 = \RR / \ZZ$.
If $\tau : \RR \to \RR$
is defined by $\tau(x) = x + 1$,
then there is an integer $e$ such that
\[
    [\tilde \alpha_1, \tilde \beta_1]
    \cdots
    [\tilde \alpha_g, \tilde \beta_g]
    =
    \tau^e.
\]
This integer $e$ is independent of the choice of lifts
and is the \emph{Euler number} of the action.
The action of a group $G$ on a topological space $X$
is \emph{minimal} if any non-empty $G$-invariant subset
is dense in $X$.

To prove Theorem \ref{leafinvariance},
it is enough to prove the following equivalent formulation.

\pagebreak[3]

\begin{proposition}\label{action}
    Suppose
    \begin{enumerate}
        \item
            $G$ is a surface group acting on $S^1$
            with non-zero Euler number,
        \item
            there is a closed $G$-invariant subset $X$ such that the action
            of $G$ on $X$ is minimal, and
        \item
            there is a homeomorphism $h : S^1 \to S^1$ that
            commutes with every element of $G$
            and such that $h(X) = X$.
    \end{enumerate}
    Then, there is $k \ge 1$ such that
    $h^k|_X$ is the identity map.
\end{proposition}

\begin{proof}
    We first suppose $h$ has irrational rotation number
    and derive a contradiction.
    In this case, $h$ is uniquely ergodic.
    As $h$ commutes with the elements of $G$,
    the unique invariant measure $\mu$
    is also invariant for every element of $G$.
    Let $\tilde \mu$ be the lift of $\mu$
    to $\RR$.
    In particular, $\tilde \mu ( [m,n) ) = n - m$
    for any integers $m < n$.
    This measure $\tilde \mu$ is invariant
    under the lifted maps $\tilde \alpha_i$
    and $\tilde \beta_i$ described above,
    and so there are real numbers $a_i, b_i$
    such that
    \[\tilde \mu \big( \, [x, \tilde \alpha_i(x)) \, \big) = a_i
    \quad \text{and} \quad
    \tilde \mu \big( \, [x, \tilde \beta_i(x)) \, \big) = b_i\]
    for all $i$ and all points $x \in \RR$.
    Since
    $\sum_i (a_i + b_i - a_i - b_i) = 0$,
    one may use this to show that the Euler number is zero.
    This gives a contradiction,
    and so we reduce to the case where $h$ has rational rotation number.

    \smallskip

    Let $k \geq 1$ be the smallest integer such that $\Fix(h^k)$
    is non-empty.
    Then any closed $h$-invariant subset of $S^1$ must intersect
    $\Fix(h^k)$.
    In particular,
    $X \cap \Fix(h^k)$
    is non-empty.
    Since $h$ commutes with all elements of $G$,
    this set
    is $G$-invariant.
    Minimality then implies that
    $X \cap \Fix(h^k) = X$.
\end{proof}

\begin{proposition}\label{fixedcover}
In the setting of Proposition \ref{leafinvariance},
if $k = 1$, that is,
if $h$ fixes every leaf of $\Lambda$,
then there is a good lift $\tilde h$
of $h$ to the universal cover $\tilde M$
such that $\tilde h$ fixes every leaf of the lifted lamination
$\tilde \Lambda$.
\end{proposition}

\begin{proof}
To prove this, we first consider an intermediate covering space $\hat M$
that is homeomorphic to $\DD \times S^1$
where $\DD$ is the Poincar\'e disc.
The circle fibering and the foliation $\mathcal F$
lift to $\hat M$.
Let $\hat{ \mathcal F}$ be the lifted foliation
and $\hat \Lambda$ be the lift of $\Lambda$.
As $h$ is isotopic to the identity,
the isotopy gives us a good lift
$\hat h : \hat M \to \hat M$
and from Proposition \ref{action}
we can see that $\hat h$ fixes every leaf of $\hat \Lambda$.

Now consider the universal covering $\tilde M \to \hat M$.
This is a infinite cyclic covering and the map
$\tau : \tilde M \to \hat M$
generating this group of deck transformations lies in the center
of the larger group of deck transformations corresponding to
the covering
$\tilde M \to M$.
We may lift $\hat h$ to a good lift $\tilde h$ defined on
$\tilde M$.
Then, there is an integer $n$ such that
$\tau^n \circ \tilde h$ fixes every leaf
of $\tilde \Lambda$.
Since $\tau$ is in the center of the group,
this composition is a good lift.
\end{proof}


\begin{remark}
In our current setting, the results of \cite{haposh} imply that the circle bundle is non-trivial.
In particular, if $f$ is partially hyperbolic and isotopic to the identity, then Theorem \ref{leafinvariance} applies to any horizontal $f$-invariant lamination. 
\end{remark}

\section{No periodic points revisited}
 
In this section, we prove Theorem \ref{maintheorem} in the case that $f$ has
no periodic points.
This proof does not rely on the announced result of R.~Saghin and J.~Yang
mentioned in the introduction and proves the stronger property of
accessibility instead of ergodicity.

The idea of the proof is as follows.
We assume that $f$ is not accessible and derive a contradiction.
First, we adapt techniques developed by J.~Zhang in the setting of
neutral center \cite{jz}
to show that our diffeomorphism satisfies the hypotheses of the
following theorem of C.~Bonatti and A.~Wilkinson.

\begin{theorem}[Theorem 2 of \cite{bowi}]\label{bowithm}
    Let $f$ be a partially hyperbolic dynamically coherent transitive
    diffeomorphism on a compact 3-manifold $M$
    and let $\Fs, \Fcs, \Fcu, \Fu, \Fc$ be the invariant foliations of $f$.
    Assume that there is a closed center leaf $\ga$
    which is periodic under $f$ and such that each center leaf in
    $W^{s}_{loc}(\ga)$ is periodic for $f$.
    Then:
    \begin{enumerate}
        \item
            there is an $n \in \NN$ such that $f^n$
            sends every center leaf to itself.
        \item
            there is an $L > 0$ such that for any $x \in M$
            the length $d_c(x, f^n(x))$ of the smaller center segment
            joining $x$ to $f(x)$ is bounded by $L$.
        \item
            each leaf $\mathcal{L}^{cu}$ of $\mathcal{F}^{cu}$
            is a cylinder or a plane
            (according to whether it contains a closed
            center leaf or not)
            and is trivially bi-foliated by $\mathcal{F}^c$
            and $\mathcal{F}^{u}$.
        \item
            the center foliation supports a continuous flow
            conjugate to a transitive expansive flow.
    \end{enumerate}
\end{theorem}

In fact, the flow is a topological Anosov flow, that is, an expansive flow without the singular orbits present in pseudo-Anosov flows. 

Consider a foliation $\mathcal F$ and its lift
to the universal cover.
The foliation is \emph{uniform}
if any two leaves of the lifted foliation are at finite Hausdorff distance.
The foliation is $\RR$-\emph{covered}
if the space of leaves of the lifted foliation is homeomorphic to $\RR$.
A flow $\phi$ is \emph{regulating} for $\mathcal F$ if $\phi$ is transverse to $\mathcal F$
and any lifted leaf intersects any lifted orbit of the flow.

In our setting, the fact that there are no periodic
points implies that there is a $us$-foliation
and we show that the flow in the center direction regulates this
foliation.
The following result of S.~Fenley then shows that $M$ cannot be a circle bundle over a higher-genus surface
and provides the needed contradiction.

\begin{theorem}[\cite{fe1,fe2}]\label{fenleythm} Let $\phi$ be a topological Anosov flow on a 3-manifold $M$ which is regulating for a uniform $\mathbb R$-covered foliation in $M$. Then, $M$ has virtually solvable fundamental group. 
\end{theorem} 

The statement of this theorem does not appear in S.~Fenley's papers but can be deduced from his arguments. For the sake of completeness, we will include a brief sketch of how the theorem is deduced from his work. 
The below sketch is based on the work in sections 4 and 5
of \cite{fe2}, but note that we are working with a topological Anosov
flow instead of a pseudo-Anosov flow
and that we know the foliation is uniform.

\begin{proof}[Sketch of the proof of Theorem \ref{fenleythm}]
%
%
    Let $\tilde \phi$ be the lift of the topological An\-osov flow $\phi$
    to the universal cover and
    let $\tilde{\mathcal F}$ be the lift of the regulated foliation.
    There are no spherical leaves
    and since \cite[Theorem 2.1]{fe2} implies the desired result
    if there is a parabolic leaf,
    we may assume that every leaf of the foliation is hyperbolic.

    Under these hypotheses,
    Fenley \cite{fe2} shows that
    the space of orbits, $\mathcal O$,
    can be identified with the Poincar\'e disk and
    each of the weak stable and unstable foliations of the lifted flow
    defines a one-dimensional foliation on $\mathcal O$.
    Call these $\tilde\Lambda^s$ and $\tilde\Lambda^u$.
    The leaves of these foliations satisfy the following properties:
\begin{enumerate}
\item They are quasigeodesic; that is, each leaf is a bounded distance from a geodesic of the Poincar\'e disk (\cite[Lemma 4.1]{fe2}).
\item \label{injectivity}
    If $l_1$ and $l_2$ are distinct leaves of $\tilde\Lambda^u$
        (or of $\tilde\Lambda^s$),
    they do not share an ideal point in $\partial \mathcal O$
    (\cite[Lemma 5.2]{fe2}).
\item Given $x\in \mathcal O$, call $l^u (x)$ to the leaf of $\tilde\Lambda^u$ through $x$. Consider an orientation of $\tilde\Lambda^u$. The ray $l^u(x)$ has two ideal points $a^u_+(x)$ and $a^u_-(x)$ corresponding to the positive and the negative direction of $l^u(x)$. The maps $x\mapsto a^u_\pm(x)$ are continuous (\cite[Lemma 5.1]{fe2}).
Analogous statements hold for $\tilde\Lambda^s$.
\end{enumerate}

The equivariance of $\tilde \phi$ implies that the fundamental group of $M$ acts on $\mathcal O$.
The action of the elements of $\pi_1(M)$ extends to homeomorphisms of $\mathcal O\cup\partial \mathcal O$.
    Consider $g\in\pi_1(M)$ with a fixed point $p = g(p) \in \cO$
    corresponding to a closed orbit of $\phi$.
    As shown by \cite[Proposition 5.3]{fe2},
    up to replacing $g$ by an iterate,
    we may assume $g$ has at least four fixed points on
    the circle $\partial \mathcal O$ at infinity.
    Moreover, these fixed points correspond exactly
    to the prongs of the stable and unstable manifolds of $p$.
    Since we are in the topologically Anosov setting
    (instead of the pseudo-Anosov setting),
    there are exactly four fixed points on $\partial \cO$
    given by
    $a^u_+(p)$,
    $a^u_-(p)$,
    $a^s_+(p)$, and
    $a^s_-(p)$.
    Under the action of $g$,
    $a^u_+(p)$ and
    $a^u_-(p)$ are attracting on $\partial \cO$
    whereas
    $a^s_+(p)$ and
    $a^s_-(p)$ are repelling.
    Again, see \cite[Proposition 5.3]{fe2} for details.

Items (2) and (3) above show that
the function $a^u_+ : \cO \to \partial \cO$ is continuous
and injective on each leaf of $\tilde\Lambda^s$, and therefore its range $A^u_+$
is an open subset of $\partial \cO$.
In particular, $A^u_+$ contains an interval centered about $a^u_+(p)$.
Since $A^u_+$ is invariant under the fundamental group,
it follows that $A^u_+$ must contain the entire basin of the
attracting fixed point 
$a^u_+(p)$ under the action of $g$.
That is, $A^u_+$ contains one of the two connected components
of $\partial \cO \setminus \{ a^s_-(p), a^s_+(p) \}$.
Similar reasoning shows that
the range $A^u_-$ of $a^u_-$ contains the other connected component
of $\partial \cO \setminus \{ a^s_-(p), a^s_+(p) \}$.
Item (2) implies that $A^u_+$ and $A^u_-$ are disjoint
and so they must exactly equal these two connected components.

As $A^u_+$ and $A^u_-$ are invariant under the action of $\pi_1(M)$,
it follows that every element of $\pi_1(M)$
fixes the set $\{a^s_-(p), a^s_+(p)\}$.
By item (2) above, this means that every element
of $\pi_1(M)$ fixes the stable leaf through $p$.
A similar analysis shows that every element fixes the unstable leaf
through $p$.
For two transverse foliations on a disk,
such an intersection must be unique,
so $p$ is fixed by every element of the fundamental group.
This is not possible for a topological Anosov flow and gives a contradiction.
\end{proof}

 

 
 
\smallskip

We now begin the detailed proof of
Theorem \ref{maintheorem} in the case that $f$ has
no periodic points.
Assume for the remainder of the section that
$M$ is a circle bundle and
$f : M \to M$ is a  volume preserving, partially hyperbolic
diffeomorphism homotopic to the identity
which is not accessible and has no periodic points.

Note that $f$ is transitive by Lemma \ref{usdense} and dynamically coherent by Theorem \ref{integrable}. As usual we define $ W^{s}(A)=\cup_{x\in A}\Fs(x)$. We say that the center stable foliation is {\em complete} if the $W^{s}(\Fc(x))=\Fcs(x)$ $\forall x\in M$. Of course we define the completeness of the center-unstable foliation in an analogous way. 

Consider a center leaf $L^{c}$ lying inside a center-stable leaf $L^{cs}$ and such that its stable saturation
$W^s(L^{c})$ is a proper subset of $L^{cs}$. 
A point $y$
is an {\em accessible boundary point} of $W^s(L^{c})$
if there exists a
curve $\gamma : [0, 1] \to M$ tangent to the center bundle such that
$\gamma([0, 1)) \subset W^{s}(L^c)$ and $\gamma(1) = y \notin W^{s}(L^c)$.
The set of such points is called the {\em accessible
boundary} of $W^s(L^c)$. Bonatti and Wilkinson \cite{bowi} showed that the accessible boundary is $s$-saturated. 

Our first step is to prove the completeness of the center-stable and center-unstable foliations.

\begin{lemma}\label{completeness} The foliations $\mathcal F^{cs}$ and $\Fcu$ are complete.
\end{lemma}

\begin{proof}
The uniform transversality of the foliations implies the existence of $\delta>0$ such that the $\delta$-neighborhood of $\Fc(x)$ in $\Fcs(x)$  is contained in $ W^s(\Fc(x))$ for all $x\in M$. 

Assume by contradiction that there is $x$ such that $  W^s(\Fc(x))\neq W^{cs}(x)$. Then, there is an accessible boundary point $y$ with respect to $  W^s(\Fc(x))$. By definition, there is a center curve $\gamma : [0, 1] \to M$ such that
$\gamma([0, 1)) \subset  W^s(\Fc(x))$ and $\gamma(1) = y$. On the one hand, we can assume that $\gamma$ is as short as desired. On the other hand, taking any $m>0$  large enough we get that the length of $f^m(\gamma)$ is greater than $\delta/2$. This is a consequence of the fact that  $f^m(y)$ is also an accessible boundary point  with respect to $  W^s(\Fc(f^m(x)))$ and $\gamma(0)\in W^s(\Fc(x))$. In fact, the same is true if we start with any iterate $f^k(y)$, we mean there is a short center curve contained in $ W^s(\Fc(f^k(x)))$ and ending at $f^k(y)$ that from a sufficiently long iterate its length is greater than $\delta/2$.

\begin{figure}
\begin{center}
\includegraphics[width=.7\textwidth]{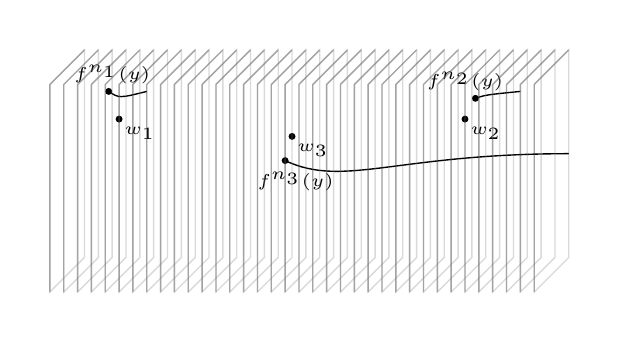}
\caption{\label{figure1} Center segments in the proof of Lemma \ref{completeness}.} 
\end{center}
\end{figure}

Take a recurrent point $z\in \omega(y)$ and let $U$ be a small foliated neighborhood of $z$ for the $us$-foliation. Observe that all iterates of $z$ belonging to $U$ are in different plaques of the foliation, otherwise Anosov closing lemma would imply the existence of a periodic point. The space of plaques in $U$ is homeomorphic to an interval, then we can order it and this induces an order on the iterates of $z$. Take three different iterates of $z$, $w_1<w_3<w_2$ such that $\dist_c (w_3, w_i)<\delta/16$, $i=1,2$,  where $\dist_c$ is the infimum of the length of center curves joining the corresponding plaques. Take now $0<n_1<n_2<<n_3$ such that $\dist_c(f^{n_i}(y), w_i)<\delta/16$. Taking $n_3$ large enough and since $f^{n_i}(y), \, i=1,2$ are accessible boundary points we have center arcs of length smaller that $\delta/16$ with end point $f^{n_i}(y), \, i=1,2$ such that after iterating $n_3-n_i, \, i=1,2$ we get the curve lengths grow larger than $\delta/2$. Now,
projecting locally via the $us$-foliation we obtain a map of the interval and as a
consequence there is a point $e$ in a $us$-leaf such that the $(n_3-n_i)$-iterate is in the same
$us$-leaf for $i=1$ or $2$. The Anosov closing lemma implies the existence of a periodic point and we arrive to a contradiction. 
\end{proof}

The center stable $\mathcal F^{cs}$ and center unstable $\mathcal F^{cu}$ foliations are Reebless \cite{bbi}. In particular, we have that all leaves of their lifts, $\tilde{\mathcal F}^{cs}$ and $\tilde{\mathcal F}^{cu}$,  to the universal cover $\tilde M$ are planes. The argument of the previous lemma applies to the universal cover and we obtain the following:

\begin{lemma}[Claim 3.2, \cite{jz}]\label{triv_foliated} 
$\tilde{\mathcal F}^{cs}$ is complete, that is, all its leaves are trivially bi-foliated by $\tilde{\mathcal F}^{s}$ and $\tilde{\mathcal F}^{c}$.
\end{lemma}

We also need the following result from \cite{hhu4}.

\begin{theorem}[\cite{hhu4}]\label{nocompact}
$\tilde{\mathcal F}^{cs}$ has no compact leaves.
 \end{theorem}
 
We include the proof of the following lemma for the sake of completeness.

\begin{lemma}[Lemma 3.3, \cite{jz}]\label{cylinder} A center stable leaf containing a compact center leaf $\gamma$ is either a cylinder or a M\"obius band.
\end{lemma}
Since we are assuming that $E^{c}\oplus E^{s}$ is orientable, in fact, all these leaves are cylinders. 

\begin{proof} 
If each stable leaf intersects $\gamma$ only once it is not difficult to show, using Lemma \ref{completeness}, that the center stable leaf containing               $\gamma$ is either a cylinder or a M\"obius band.

If there is a stable leaf intersecting $\gamma$ twice we take a lift of the center stable leaf to the universal cover. We have at least two lifts of $\gamma$ in it. 
Completeness, Lemma \ref{triv_foliated}, gives that the image of the band limited by these two lifts is the whole center stable manifold and it is compact. We arrived to a contradiction. 
\end{proof}

In order to move forward we need the following general property of  center stable foliations. 

\begin{proposition}\label{minimality} Let $f$ be a dynamically coherent partially hyperbolic diffeomorphism such that $\Omega(f)=M$. Then $\mathcal F^{cs}$ is minimal.
\end{proposition}

\begin{proof} According to Theorem \ref{nocompact}, $\mathcal F^{cs}$ has no compact leaves. Thus, it has a finite number of minimal sets \cite[Theorem V.4.1.3]{hh}. 
The union of its minimal sets $\Lambda$ is a compact invariant set and partial hyperbolicity implies that it is a repeller. Since $\Omega(f)=M$ the only possibility is that $\Lambda =M$, and then $\mathcal F^{cs}$ is minimal.
\end{proof}

\begin{lemma}\label{horizontal.cs} $\mathcal F^{cs}$ is horizontal and there is $k\geq 1$ such that $f^{k}$ fixes all center-stable leaves. Moreover, this property lifts to the universal cover. 
\end{lemma}

\begin{proof}  The set of vertical leaves is closed, if it is not empty, then minimality implies that all leaves are vertical. This contradicts Remark \ref{foliation.with.vertical.leaves}. 
Horizontality is also proven in \cite{haposh}. 
Now we can apply the results in Section \ref{section.horizontal} to get that all leaves are periodic. 
\end{proof}

We shall assume from now on that all leaves are fixed.

\begin{lemma}\label{cs.plane.cylinder}
Every center-stable leaf is either a plane or a cylinder. 
 \end{lemma}
 
 \begin{proof}  By Lemma \ref{cylinder} we need only consider a leaf $W$ which does not contain a center circle. 
  As in \cite{jz}, we use an argument from \cite{bowi}. 
 Let $\tilde W$ be a lift of $W$ to the universal cover and $G$ be the group of deck transformations leaving invariant $\tilde W$. $G$ acts in both, the spaces of center and stable leaves of $W$. Both actions are without fixed points because there are neither compact center leaves nor compact stables leaves. Since both spaces are homeomorphic to lines, we have that both actions are orientation preserving and abelian \cite[Theorem 6.10]{ghys}. Since $G$ is a subgroup of the product of these  two actions we obtain that $G$ is abelian. Thus $W$ is a non-compact orientable surface with abelian fundamental group. This implies that $W$ is either a plane or a cylinder.
 \end{proof}
 
 \begin{lemma}\label{compact.center} Suppose  that a center-stable leaf $W$ is a cylinder. Then it contains an $f$-invariant closed center curve. \end{lemma}
 
 \begin{proof} Due to the completeness of $\tilde W$ and the fact that $G$ in the previous proof  is non-trivial, if $W^{c}(x)$ is a non-compact center leaf with $x\in W$, then $W^{c}(x)$ intersects $W^s(x)$ in at least two points.
 
 Now, take $\alpha$ a simple closed curve formed by the union of an arc joining to consecutive intersections of $W^c(x)$ with $W^s(x)$ and an arc of $W^s(x)$. The curve $\alpha$ intersects all stable manifolds in  $W$ and the same happens with $f^p(\al)$. Compactness of $\al$ implies that the length of the stable arcs joining $\al$ with $f(\al)$ is uniformly bounded. This implies that there is $C>0$ large enough such that $\overline{f^p(W^s_C(\al))}\subset W^s_C(\al)$. The uniform contraction along stable manifolds implies that $K=\bigcap_{n>0}f^{n}(W^s_C(\al))$ intersects each stable manifold in exactly one point. Moreover, iterating $\alpha$, it is clear that $K$ contains a center curve through each point and then, it is a fixed closed center curve. This proves the lemma. 
\end{proof}

At this point we can apply Theorem \ref{noperiodicpoints} to get the ergodicity of $f$.
However, we continue in order to prove accessibility. 
 
\begin{lemma}\label{intersection.center}
Given $\tilde x,\tilde y\in\tilde M$, $\tilde\Fcs(\tilde x)$ and $\tilde \Fcu(\tilde y)$ intersect in at most one center curve. 
\end{lemma}

\begin{proof}
If the intersection contains two center curves, completeness implies the existence of a stable curve intersecting $\tilde \Fcu(\tilde y)$ in two points. This is a contradiction. 
\end{proof}

\begin{lemma}\label{center.invariance}
There is a lift $\tilde f$ of $f$ such that all center leaves are $\tilde f$-invariant.
\end{lemma}

\begin{proof}

By Lemma \ref{horizontal.cs}, there exists $\tilde f$ leaving invariant all leaves of $\tilde{\mathcal F}^{cs}$. Since $M$ is not the 3-torus, at least one leaf is not a plane and Lemma \ref{compact.center} implies the existence of a compact center leaf $\gamma$. Let $\tilde\gamma$ be any lift of $\gamma$. Since $W^s(\tilde \gamma)$ is $\tilde f$-invariant and the the lift of $\gamma$ has a unique component in $W^s(\tilde \gamma)$, then $\tilde\gamma$ is $\tilde f$-invariant. Then we get the $\tilde f$-invariance of $W^u(\tilde \gamma)$. As a consequence of both Lemma \ref{intersection.center} and the invariance of the center stable leaves, all center leaves in $W^u(\tilde \gamma)$ are $\tilde f$ invariant. Minimality of the center unstable foliations implies that the center unstable leaves containing a lift of $\gamma$ are dense, so we have the $\tilde f$-invariance for a dense set center manifolds. Continuity of $\tilde f$ implies the lemma. 
\end{proof}

We have now established all of the hypothesis of Theorem \ref{bowithm}.
Let $\Phi$ be the resulting topologically expansive flow along the center
direction.
%
%
%
%
%
Since the stable and unstable foliations of $\Phi$ coincide with the center-stable and center-unstable foliations of $f$ we have that $\Phi$ is a topological Anosov flow. We can also derive this  from Brunella \cite{bru} who proved this is always the case for any expansive flow on a Seifert manifold. 

%

 \begin{lemma}\label{regulating}
 $\Phi$ is regulating for $\mathcal F^{us}$.
\end{lemma}

\begin{proof}
    For this proof, we work entirely on the universal cover.
    For a point $x \in \tilde M$, define
    \begin{enumerate}
        \item
            $J_x$ as compact center segment
            between $x$ and $\tf(x)$,
        \item
            $U_x$ as the union of all $us$-leaves
            which intersect $J_x$, and
        \item
            $V_x$ as the union of all $us$-leaves
            which intersect the center leaf
            containing $x$.
    \end{enumerate}
    As $\tf$ has no fixed points,
    $
        \tFc(x) = \cup_{n \in \ZZ} \tf^n(J_x)
    $
    and so
    $
        V_x = \cup_{n \in \ZZ} \tf^n(U_x).
    $
    If $x$ and $y$ are sufficiently close points
    on the same $us$-leaf,
    then $J_x$ and $J_y$ are nearby center segments
    with endpoints lying on the same two $us$-leaves.
    It follows that $U_x = U_y$
    and therefore $V_x = V_y$.
    Clearly, if points $y$ and $z$ are on the same
    center leaf,
    then $V_y = V_z$.
    Since any two points on the universal cover
    may be connected by a concatenation
    of stable, unstable, and center curves,
    $V_x$ is independent
    of the point $x$
    and so $V_x = \tM$ for all $x$.
    This proves that the flow in the center direction
    is regulating.
\end{proof}

Since the $us$-foliation is a horizontal foliation
on a circle bundle, one can verify that it is uniform and $\RR$-covered.
Theorem \ref{fenleythm} then gives a contradiction.

%
%
%
%

 \section{Proof of the main theorem}
 
The previous section proved Theorem \ref{maintheorem}
in the case of no periodic points.
We now handle the case of periodic points.
To do this, we will use of the following result of Mendes.
 
 \begin{theorem}\cite{mendes}\label{mendes}
 If $S$ is homeomorphic to the plane, then an Anosov map on $S$
 can have at most one fixed point.
 \end{theorem}
 
 \begin{proof}[Proof of Theorem \ref{maintheorem}]
By the work of the previous section, we may assume that $f$ has periodic points.
Assume $f$ is not accessible.
Then, there is a $us$-lamination $\Ga(f)$.
Let $\La$ be a minimal subset of $\Ga(f)$.
Since $\Ga(f)$ may be completed to a foliation without compact leaves,
it has at most finitely many minimal sets \cite[Theorem V.4.1.3]{hh}.
Hence, up to replacing $f$ by an iterate, we may assume $\La$ is $f$-invariant.
By Proposition \ref{ushorizontal} and Theorem \ref{leafinvariance},
we may again replace $f$ by an iterate and assume that every leaf of $\La$
is fixed by $f$.
If $\La = M$, then it contains periodic points by assumption.
Otherwise, Proposition \ref{sublamination} implies that $\La$
contains periodic points.
In either case, let $p$ be a periodic point in $\La$.
Again taking an iterate, assume $p$ is fixed.

By Proposition \ref{fixedcover},
there is a good lift $\tilde f$ of $f$ to the universal cover
such that every leaf of the lifted lamination is invariant by $\tilde f$.
Let $\tilde p$ be a lift of $p$ and assume it lies on a leaf
$\tilde L$ of the lifted lamination.
Then there is a deck transformation $\tau$ such that
$\tau \tilde f (\tilde p) = \tilde p$.
As $\tilde f$ fixes $\tilde L$, $\tau$ also fixes $\tilde L$.
Since $\tau$ commutes with $\tau \tilde f$,
$\tau(\tilde p)$ is also a fixed point for $\tau \tilde f$.
Then Theorem \ref{mendes} implies that $\tau(\tilde p) = \tilde p$
and so $\tau$ is the identity map.


We have reduced to the case where $\tilde f(\tilde p) = \tilde p$.
As $\La$ is minimal,  the leaf $L$ through $p$ self-accumulates and there is a lift $\tilde p'$ different from $\tilde p$ lying very close to $\tilde L$.  Observe that $\tilde p'$ is also fixed by  $\tilde f$. According to \cite{bdu} there is an $\tilde f$-invariant center curve $\al$ through $\tilde p'$. We may assume $\tilde p'$ was chosen close enough that $\al$ intersects $\tilde L$.
The lamination $\Ga(f)$ may be completed to a Reebless foliation on $M$
and we may assume that $\al$ is transverse to the lift of this foliation.
Novikov's Theorem then implies that $\al$ intersects $\tilde L$ at most once. Since both $\al$ and $\tilde L$ are $\tilde f$-invariant, their intersection is a fixed point. This fixed point $\tilde L$ is distinct from $\tilde p$
and Theorem \ref{mendes} gives a contradiction.
\end{proof}
 
\section{Small Seifert manifolds} \label{sec.turnover}

In this final section, we prove Theorem \ref{turnover}.
Let $f$ be as in the statement of the theorem.
As explained in \cite{haposh},
$f$ induces a base map $\sigma : \Sigma \to \Sigma$
which is defined up to isotopy,
the orbifold is $\Sigma$ is finitely covered by a hyperbolic surface
and the map $f$ may be lifted to a map $f_1$
defined on the unit tangent bundle of this surface.
As the mapping class group of the orbifold is finite
\cite[Proposition 2.3]{farbmarg},
there is an iterate $f^k$ such that
the base map $\sigma^k$ of $f^k$ is isotopic to the identity.
Replacing $f$ by an iterate, we assume $k = 1$.
Then $f_1$ also has a base map isotopic to the identity
and \cite[Proposition 3.5]{bghp} implies that $f_1$ itself
is isotopic to the identity.
(Note: the proof in \cite{bghp} implicitly uses a theorem
of Matsumoto to show that a closed geodesic may be isotoped
to lie in a leaf of the branching center-stable foliation.)
Accessibility then follows from Theorem \ref{maintheorem}.


\end{document}